\documentclass[12pt, twoside]{article}
\usepackage{amsmath,amsthm,amssymb}
\usepackage{times}
\usepackage{enumerate}

\def\titlerunning#1{\gdef\titrun{#1}}
\makeatletter
\def\author#1{\gdef\autrun{\def\and{\unskip, }#1}\gdef\@author{#1}}
\def\address#1{{\def\and{\\\hspace*{18pt}}\renewcommand{\thefootnote}{}%
\footnote {#1}}%
\markboth{\autrun}{\titrun}}
\makeatother
\def\email#1{e-mail: #1}
\def\subjclass#1{{\renewcommand{\thefootnote}{}%
\footnote{\emph{Mathematics Subject Classification (2010):} #1}}}

\def\Ld{{\mathcal{L}(\mathbb{C}^d,\mathbb{C}^d)}}

\def\P{{\mathbb{P}}}

\def\be{\begin{equation}}
\def\ee{\end{equation}}

\def\ba{{\begin{align}}}
\def\ea{{\end{align}}}

\def\GL{{\mathrm {GL}}}
\def\gl{{\mathfrak{gl}}}

\def\0{{\mathbf 0}}

\def\cal{\mathcal}

\def\SL{{\mathrm {SL}}}

\newtheorem{thm}{Theorem}[section]
\newtheorem{cor}[thm]{Corollary}

\newtheorem{lemma}[thm]{Lemma}

\theoremstyle{remark}

\numberwithin{equation}{section}

\theoremstyle{definition}

\def\note#1
{\marginpar
{\tiny $\leftarrow$
\par
\hfuzz=20pt \hbadness=9000 \hyphenpenalty=-100 \exhyphenpenalty=-100
\pretolerance=-1 \tolerance=9999 \doublehyphendemerits=-100000
\finalhyphendemerits=-100000 \baselineskip=6pt
#1}\hfuzz=1pt}

\def\tr{{\text{tr}}}

\newcommand{\dist}{\operatorname{dist}}

\newcommand{\id}{\operatorname{id}}

\newcommand{\DD}{{\cal D}}

\newcommand{\LL}{{\cal L}}
\newcommand{\MM}{{\cal M}}

\newcommand{\OO}{{\cal O}}

\newcommand{\UU}{{\cal U}}
\newcommand{\VV}{{\cal V}}

\newcommand{\C}{{\mathbb C}}
\newcommand{\D}{{\mathbb D}}
\newcommand{\E}{{\mathbb E}}

\newcommand{\N}{{\mathbb N}}
\newcommand{\Q}{{\mathbb Q}}
\newcommand{\R}{{\mathbb R}}
\newcommand{\T}{{\mathbb T}}

\newcommand{\Z}{{\mathbb Z}}

\def\B0{{\bold{0}}}

\catcode`\@=12

\def\Empty{}
\newcommand\oplabel[1]{
  \def\OpArg{#1} \ifx \OpArg\Empty {} \else
  	\label{#1}
  \fi}
		
\newcommand{\comm}[1]{}
\newcommand{\comment}[1]{}

\newcommand{\qtx}[1]{\quad\text{#1}\quad}
\def\Tr{{{\rm Tr}}}
\def\im{{\rm Im}}
\def\re{{\rm Re}}

\DeclareMathOperator{\opCap}{\mathrm{Cap}}

\begin{document}


\baselineskip=17pt


\titlerunning{Complex one-frequency cocycles}

\title{Complex one-frequency cocycles}

\author{Artur Avila \and Svetlana Jitomirskaya \and Christian Sadel}

\date{}

\maketitle

\address{A. Avila: 
CNRS, IMJ-PRG, UMR 7586, Univ Paris Diderot, Sorbonne Paris Cit\'e,
Sorbonnes Universit\'es, UPMC Univ Paris 06, F-75013, Paris, France \&
IMPA, Estrada Dona Castorina 110, Rio de Janeiro, Brasil;
\email{artur.avila@imj-prg.fr}
\and
S. Jitomirskaya: 
University of California, Irvine, 340 Rowland Hall, Irvine
CA 92697-3875,  USA;
 \email{szhitomi@math.uci.edu} \\
 \thanks{The work of S.J. was supported by the NSF Grant DMS 1101578}
\and
C. Sadel: 
University of California, Irvine, 340 Rowland Hall, Irvine
CA 92697-3875,  USA \&
University of British Columbia, 1984 Mathematics Road, Vancouver, BC, \mbox{V6T\,1Z2}, Canada;
 \email{csadel@math.ubc.ca} \\
 \thanks{The work of C.S. was supported by the NSERC Discovery grant 92997-2010 RGPIN}
}


\begin{abstract}
We show that on a dense open set of analytic one-frequency complex valued
cocycles in
arbitrary dimension Oseledets filtration is either {\it dominated} or
{\it trivial}.  The underlying mechanism is different from that of the
Bochi-Viana Theorem for continuous cocycles, which links non-domination with
discontinuity of the Lyapunov exponent.  Indeed, in our setting
the Lyapunov exponents are shown to depend continuously on the cocycle, even
if the initial irrational frequency is allowed to vary.
On the other hand, this last property provides a good control of the
periodic approximations of a cocycle, allowing us to show
that {\it domination}
can be characterized, in the presence of a gap in the Lyapunov spectrum,
by additional regularity of the dependence of sums of
Lyapunov exponents.
\end{abstract}

\setcounter{tocdepth}{1}


\tableofcontents

\setcounter{footnote}{0}

\section{Introduction}

In dynamical systems and ergodic theory, the fundamental
notion of hyperbolicity appears under many guises, which are usually
split into two broad categories.
Generally speaking, {\it
uniform} notions of hyperbolicity play a major role in the description of
robust behavior.  For instance, the strongest such notion is called simply
{\it uniform hyperbolicity} and is closely associated to
structural stability, while
a weakening of this concept, {\it partial hyperbolicity}, has been
intensively developed in particularly for its connection with stable
ergodicity.
The weakest form of uniform hyperbolicity, sometimes called projective
hyperbolicity, demands merely the presence of a continuous {\it dominated}
decomposition
of the tangent dynamics, and has been linked to robust transitivity as well as
robustness of positive entropy.

On the other hand, {\it nonuniform} notions of hyperbolicity are developed
around the Oseledets Theorem, which provides a decomposition of
the tangent dynamics at almost every point with non-trivial Lyapunov
spectrum.  Of course, such a decomposition is, {\it a priori} only
measurable, and it may depend wildly on parameters, but the
flexibility afforded by getting rid of continuity requirements makes for
much greater potential applicability.  For instance, while there are
manifolds (such as even dimensional spheres) that do not support any
non-trivial continuous decomposition of the tangent
bundle, any manifold supports ergodic
non-uniformly hyperbolic conservative dynamics \cite {DP}.

In his address at the 1983 ICM \cite {M},
Ma\~n\'e suggested that the apparent gap
between uniform and nonuniform notions of hyperbolicity can be bridged in
the case of generic conservative
dynamical systems in the $C^1$-topology.  This program
was eventually developed by Bochi-Viana \cite {BV}, who proved that for
almost every orbit, either all Lyapunov exponents are zero or the
Oseledets splitting is dominated, and hence either
there is no hyperbolicity at all (even nonuniform), or uniform projective
hyperbolicity takes place.  Moreover, those results were also obtained in
the setting of continuous
cocycles over measure-preserving transformations.

In full generality, the Bochi-Viana Theorem is certainly
dependent on low
regularity considerations: For instance, there are open sets of
(sufficiently smooth) ergodic conservative diffeomorphisms for which the
Oseledets splitting
is not dominated.  However, as far as we know, all such examples currently
rely on some underlying uniform form hyperbolicity (see, e.g., \cite {AV},
\cite {ASV}).

It would seem that similar considerations apply
to the case of cocycles over
hyperbolic transformations: Indeed,
non-zero Lyapunov exponents tend to appear
robustly already for H\"older regularity, even in the presence of
topological obstructions to domination \cite {V}.  But we will show in
this paper that Ma\~n\'e's
picture turns out to hold unexpectedly in very large regularity in one
important setup.

\subsection{Bochi-Viana Theorem for analytic one-frequency complex cocycles}

Let $\Ld$ denote the set of linear operators from $\C^d$ to $\C^d$,
i.e. the set of $d\times d$ complex matrices. A complex
one-frequency cocycle is given by a pair $(\alpha,A)$, where
$\alpha\in \R$ is the frequency and $A\in C^0(\R/\Z,\Ld)$
is a continuous function from
$\R/\Z$ to $\Ld$, understood as a map
$(\alpha,A):(x,w) \mapsto (x+\alpha,A(x) \cdot w)$.
The cocycle iterates are given by $(\alpha,A)^n=(n \alpha,A_n)$, where the
$A_n$ are given by
\be
A_n(x)=\prod_{j=n-1}^0 A(x+j\alpha).
\ee
If we want to emphasize the dependence on the frequency,
then we write $A_n(\alpha,x)$.
We will be mostly interested
in the case of irrational frequencies, $\alpha\in\R\setminus\Q$.  In this
case, the dynamics is ergodic and
the Oseledets Theorem provides us with a strictly decreasing sequence of
Lyapunov exponents
$\gamma_j \in [-\infty,\infty)$ of multiplicity
$m_j \in \N$, $1 \leq j \leq
k \leq d$, so that $\sum_{j=1}^k m_j=d$ and for almost every
$x \in \R/\Z$ there
exists a filtration $\C^d=\tilde E^1_x \supset \cdots \supset \tilde E^k_x$
with $\dim
\tilde E^j_x=m_j+\cdots+m_k$, depending measurably on $x$,
that is invariant in the sense that
$A(x) \cdot \tilde E^j_x \subset \tilde E^j_{x+\alpha},\, j=1,\dots,k,$ \footnote{if $A(x)$ is invertible, then
$A(x) \cdot \tilde E^j_x = \tilde E^j_{x+\alpha}$, if $A(x)$ has a kernel, then $\ker A(x)\subset \tilde E^k_x$}
and for every
$w \in \tilde E^j_x \setminus\tilde E^{j+1}_x$
 we have $\limsup_{n \to \infty} \frac {1} {n} \ln \|A_n(x) \cdot
w\|=\gamma_j$. Such a filtration often (always, in the invertible
case) is associated with an invariant\footnote{ In the sense that $A(x) \cdot 
  E^j_x = E^j_{x+\alpha},\, j=1,\dots,k-1,\; A(x) \cdot 
  E^k_x\subset E^k_{x+\alpha}$}  continuous decomposition 
$\C^d=E^1_x \oplus \cdots \oplus E^k_x$ 
with $\dim E^j_x=m_j$ and $\tilde E^j_x = E^j_x \oplus \cdots \oplus E^k_x,$ also depending measurably on $x$, with 
$\limsup_{n \to \infty} \frac {1} {n} \ln \|A_n(x) \cdot w\|=\gamma_j$ for every $w \in E^j_x \setminus\{0\}$ \cite{R}.

An invariant continuous
decomposition $\C^d=E^1_x \oplus \cdots \oplus E^k_x$ is called dominated if there
exists $n \geq 1$ such that for any unit
vectors $w_j \in E^j_x$ we have $\|A_n(x) \cdot w_j\|>\|A_n(x) \cdot
w_{j+1}\|.$  It can be shown that such  a dominated
decomposition is robust, in the sense that small perturbations of the
cocycle will still display a dominated decomposition which will be a
small perturbation of the original one. We will say that a filtration
is dominated if it is associated with a dominated decomposition.

The Bochi-Viana Theorem, specified to this setting, establishes that for
each $\alpha \in \R \setminus \Q$, there exists a residual subset of
$A \in C^0(\R/\Z,\GL(d,\C))$ such that the Oseledets splitting is dominated.  Our
main result shows that even a significantly stronger statement is true in the analytic category.

\def\cV{{\mathcal{V}}}

\begin{thm}\label{main}

Fix $\alpha \in \R \setminus \Q$.  There exists a dense open
subset $\cV \subset C^\omega(\R/\Z,\Ld)$ such that for every $A \in \cV$ the
Oseledets filtration of $(\alpha,A)$ is either trivial\footnote{We do not know whether the set with trivial Oseledets filtration (all Lyapunov exponents are equal) contains 
 an open set or not within the set of analytic complex cocycles.} or 
dominated.

\end{thm}

Here we endow
the space $C^\omega(\R/\Z,\Ld)$ with the usual inductive limit topology.  We
will actually show a somewhat stronger version of this result, namely, where
$C^\omega(\R/\Z,\Ld)$ replaced
by a Banach space $C^\omega_\delta(\R/\Z,\Ld)$ of
analytic functions $A:\R/\Z \to \Ld$ admitting a holomorphic extension to
$\{|\im z|<\delta\}$ which is continuous up to the boundary.

\subsection{Regularity and domination}

The proof of the Bochi-Viana Theorem given in \cite {BV} centers around the
idea that an absence of domination in the Oseledets splitting can
be exploited to ``mix'' different Lyapunov exponents through suitable
perturbations, and hence it leads to discontinuity of the Lyapunov spectrum.
On the other hand, a very general Baire category
reasoning guarantees that
the Lyapunov exponents must be continuous at a {\it generic} cocycle.

At a {\it very} rough level, something similar is taking place in our setting, in
that we do show that some (verified on an open and dense set)
regularity of the dependence of
Lyapunov exponents with respect to parameters implies domination (or
triviality) of the Oseledets splitting.  The actual details are however
completely different, starting with the fact that the regularity property
which is related to domination is not merely continuity of the Lyapunov
exponent.  In fact, it turns out to involve the holomorphic
extension of the cocycle dynamics, and was first introduced (in the
particular case of $\SL(2,\C)$-cocycles)
by Avila in \cite {Avi}.

Let $L_1(\alpha,A)
\geq ... \geq L_d(\alpha,A)$ be the Lyapunov exponents of
$(\alpha,A)$ repeated according to their multiplicity, i.e.,
\be
L_k(\alpha,A) =
\lim_{n\to\infty} \frac1n \int_{\R/\Z} \ln(\sigma_k(A_n(\alpha,x))) dx,
\ee
where for a matrix $B\in\Ld$ we denote by
$\sigma_1(B)\geq\ldots\geq\sigma_d(B)$ its singular values
(eigenvalues of $\sqrt{B^*B}$).  Since the $k$-th exterior product
$\Lambda^k B$ of $B$ satisfies
$\prod_{j=1}^k \sigma_j(B)=\sigma_1(\Lambda^k B)=\|\Lambda^k B\|$,
$L^k(\alpha,A)=\sum_{j=1}^k
L_j(\alpha,A)$ satisfies
\be
L^k(\alpha,A)=L_1(\alpha,\Lambda^k A)=\lim_{n \to \infty} \frac {1} {n}
\int_{\R/\Z} \ln \|\Lambda^k A_n(\alpha,x)\| dx.
\ee
By analyticity, one can extend the function $A(x)$ to a strip
$|\im\,x|<\delta$ in the complex plane.
Then, by subharmonicity and constancy in $\re\, x$,
$L^k(\alpha,A(\cdot+it))=L_1(\alpha,\Lambda^k A(\cdot+it))$ is a
{\it convex} function of $t\in(-\delta,\delta)$\footnote{This can be
  viewed as a corollary of Hadamard's three-circle theorem} unless it is 
identically\footnote{Convexity implies that right-derivatives exist and the graph lies above the tangent line.
Hence, $L^k(\alpha,A(\cdot+it))$ is either always or never $-\infty$.}
equal to $-\infty$.  We say that $(\alpha,A)$ is
{\bf $k$-regular} if $t \mapsto
L^k(\alpha,A(\cdot+it))$ is an {\it affine} function of $t$
in a neighborhood of $0$.

Let us say that $(\alpha,A)$ is {\bf $k$-dominated} (for some $1 \leq k \leq d-1$)
if there exists a dominated
decomposition $\C^d=E^+ \oplus E^-$ with $\dim E^+=k$.
If $\alpha \in \R \setminus \Q$, then it follows from the definitions
that the
Oseledets splitting is dominated if and only if
$(\alpha,A)$ is $k$-dominated for each $k$ such that
$L_k(\alpha,A)>L_{k+1}(\alpha,A)$.

The next two results give the basic relation
between regularity and domination and show that regularity is fairly
frequent.

\begin{thm}
\label{thm-i}
Let $\alpha \in \R \setminus \Q$, $A \in C^\omega(\R/\Z,\Ld)$.  If $1
\leq k \leq d-1$ is such that $L_k(\alpha,A)>L_{k+1}(\alpha,A)$ then
$(\alpha,A)$ is $k$-regular if and only if $(\alpha,A)$ is
$k$-dominated.

\end{thm}

\begin{thm}
\label{thm-ii}
Let $\alpha \in \R \setminus \Q$, $A \in C^\omega(\R/\Z,\Ld)$ and $L_k(\alpha,A)>-\infty$.  
Then for every $t \neq 0$ small enough, $(\alpha,A(\cdot+it))$ is $k$-regular.

\end{thm}

The last result means that the convex
functions $t \mapsto L^k(\alpha,A(\cdot+it))$ are in fact piecewise
affine.  As in \cite {Avi}, this behavior is connected to a quantization
phenomenon which we now describe.  If $L^k(\alpha,A) \neq -\infty$,
define the accelerations
\be
\omega^k = \lim_{\epsilon\to 0^+} \frac{1}{2\pi\epsilon}
\left(L^k(\alpha,A(\cdot+i\epsilon))
-L^k(\alpha,A) \right)\;,\quad
\omega_k=\omega^k-\omega^{k-1}.
\ee
It is easy to see that $\omega^k$ is an integer for $k$-dominated cocycles
(also, $\omega^d$ is always an integer if $L^d(\alpha,A) \neq -\infty$).
The next result shows that this topological phenomenon manifests also
in the general case:

\begin{thm}\label{thm-iii} Let $\alpha \in \R \setminus \Q$,
$A\in C^\omega(\R/\Z,\Ld)$.  Then
the acceleration is quantized: there exists $1\leq l\leq d$, $l\in\N$,
such that $l\omega^k$ and $l\omega_k$ are integers.\footnote{ We note
  that Theorems \ref{thm-ii}, \ref{thm-iii} do not in general hold for $\alpha\in \Q,$
(see a simple counterexample in Remark 5 in \cite{Avi}).}
If $A\in C^\omega(\R/\Z,\SL(d,\C))$ then $1\leq l \leq d-1$.
\end{thm}

Theorems \ref{thm-i}-\ref{thm-iii} generalize earlier
results \cite{Avi} (also extended in \cite{jm}) obtained for $d=2.$

At this point, we must note a fundamental distinction between the analytic
and the continuous setups.
The Bochi-Viana Theorem (specified to cocycles over irrational translations)
is proved by showing that
if $L^k>L^{k+1}$ and $(\alpha,A)$ is not $k$-dominated then $A$ is not a
continuity point of $L^k$ on $C^0(\R/\Z,\Ld)$.
It turns out that for analytic cocycles, $L^k$ is continuous everywhere. 
Moreover, we may even perturb the {\it frequency,} and this indeed plays a
fundamental role in our analysis.

\begin{thm} \label{thm-cont}
The functions $\R\times C^\omega(\R/\Z,\Ld)\ni
(\alpha,A)\mapsto L_k(\alpha,A)\in[-\infty,\infty)$
are continuous at any $(\alpha',A')$ with $\alpha'\in \R\setminus\Q$.
\end{thm}

Theorem \ref{thm-cont} is optimal in that $L_k$ can be
discontinuous at rational frequencies\footnote{They are for the
  almost Mathieu cocycles $A(x)=\begin{pmatrix} E-\lambda\cos 2\pi x & -1\\ 1
  & 0 \end{pmatrix}$ as follows from \cite{K},  or see an example in the
  Remark 5 of \cite{Avi} } or in lower (even $C^\infty$) regularity \cite{WY,jm}.

This result generalizes earlier results \cite {BJ}, \cite {jm} for the case
$d=2$.  Bourgain \cite{B} obtained also joint continuity for
non-singular $d=2$ cocycles over
rotations of higher-dimensional tori. The extension to higher
$d$ has been open for over a decade.

 A somewhat related theme are  quantitative continuity results for (mainly, non-singular) analytic cocycles
 with a {\it fixed} Diophantine frequency (\cite {GS} (for $\SL(2,\R)$)) more
 recently extended  to $\GL(d,\C)$ in \cite{Sch,DK}).\footnote{It should be noted that the results of the present paper
  preceded the independent recent work \cite{Sch,DK},
(e.g. \cite{J})} Those also hold
 for the multi-frequency case. 

However we are particularly concerned with the dependence
on the frequency (and especially the behavior of the Lyapunov exponents of
rational approximations); among other things it is the key ingredient in the proofs of
Theorems 1.1-1.4.

We also note that all the other past and recent continuity results: both \cite{BJ,B,jm}
that provide joint continuity in the cocycle and frequency and \cite{GS,Sch,DK} that are for a
fixed Diophantine frequency, are based on some form of the Avalanche
principle (originally in \cite{GS}) and large deviation theorem (see \cite{bbook}).  
Here we develop a different strategy which is indeed
intimately related to the proof of the connection of regularity and
domination: it focuses on the direct construction of invariant cone fields
for certain complex phases.  This allow us to cover all irrational
frequencies without the need
to delve into arithmetic considerations.

Our approach of selecting complex phases for which such an analysis can be
carried out is 
ideologically close to the new proof of the result of \cite{BJ}
developed in the appendix of \cite{B}.  On the technical level,
our key analytic argument,
given in Section \ref{brown}, borrows some important ingredients
from \cite{A}.

Finally, the extension of various continuity results originally
obtained for $\SL(2,\C)$ to the singular case has been achieved
gradually, by overcoming a significant number of technical challenges
\cite{JKS,T,jm2,jm}. In our current approach  singularity of cocycles does not
present an additional difficulty.

{\bf Acknowledgement:} We would like to thank Christoph Marx and
Adriana Sanchez for their useful comments on this manuscript.

\section{A Brownian motion argument}\label{brown}
A quick motivation for the main Theorem of this section (which is of
general nature) is the following. Consider $\psi=\ln |P(x)|$ where $P$
is a trigonometric polynomial of degree $n$. Then, by the Lagrange
interpolation trick, that has been used in the proofs of localization
for the almost Mathieu operator, for any $\epsilon>0,$ $\psi(x)$ cannot
be smaller than $\sup \psi(x) -\epsilon$ at $n+1$ uniformly distributed
points, for large $n.$ The same cannot be said of course about an arbitrary
subharmonic function. A tool that has been used in the proofs of
localization for analytic potentials and continuity arguments is Large
Deviation Theorems, showing that ``almost invariant'' subharmonic
functions deviate from the mean only on sets of small measure. In the
present argument
the key idea is that complexifying the argument leads to many values
of the imaginary part where the situation is as nice as for the $n-$th
degree polynomial.

We start with what we call the Big Obstacle Lemma.

\begin{lemma}
\label{lemma-1}
There exists $c>0$ with the following property.
Let $B \subset \R^2$ be a Borel set with non-empty intersection with
$(-1,1) \times \{t\}$ for a subset of $t \in (-1,1)$ of Lebesgue measure at
least $\rho$.  Run Brownian motion
starting at the origin until it escapes from $(-2,2) \times (-2,2)$.  Then
the probability of hitting $B$ before escape is at least $c \rho$.

\end{lemma}

\begin{proof}
We will first need some notions from potential theory.
Let $\mu$ be a continuous probability measure supported on a Borel subset
$A\subset \R^2$. Given a kernel $K$ i.e. a measurable function
$K:\R^2\times\R^2\to [0,\infty]$ such that $K(x,y) $ is a continuous
and decreasing function of $|x-y|$, one can define the energy of
$\mu,\; $ with respect to $K$ by 
\begin{equation}\label{energy}I_K(\mu)=\int_A\int_A
K(x,y)d\mu(x)d\mu(y).\end{equation}
 The corresponding capacity $Cap_K(A)$ is
defined as $\frac1{\inf_\mu I_K(\mu)}.$ Note that the standard
logarithmic capacity which we denote $\opCap (A)$ is defined differently
than $\opCap_{-\ln |x-y|}(A)$, namely,  $\opCap (A)=e^{-\inf_\mu I_K(\mu)}.$
The $\inf$ in (\ref{energy}) is achieved at the unique measure. In
case of the  logarithmic kernel  it is called
the equilibrium measure for $A.$ 

We will need two kernels: the Green kernel $G(x,y)$ and the Martin
kernel $M(x,y)$. Here $G$ is the
Green's function of Brownian motion stopped at exit from  $(-2,2) \times (-2,2)$, namely

\begin{equation}\label{green}
G(x,y) = -\frac 1\pi(\ln |x-y| -\E_x \ln |B(T) -y|)
\end{equation}
 where $B(T)$ is the Brownian motion at time of first
 hit of the boundary of $(-2,2) \times (-2,2)$  and $\E_x$ stands for
 the expectation over Brownian motion started at $x.$ $M(x,y)$ is
 defined as $\frac{G(x,y)}{G(0,y)}$ for $x\not= y$ and $M(x,x)=\infty.$

We will use that for compact sets $A \subset (-2,2) \times (-2,2),$ 
$$\P\{\mbox{hitting}\;\, A\;\, \mbox{before escape from}\;\, (-2,2) \times (-2,2)\} \geq
1/2\, \opCap_M (A)$$ (see \cite{mp}, Theorem 8.24).
Since $G(0,y)>c>0,$ this implies that for any probability measure $\mu,$
$I_M(\mu)<cI_G(\mu).$ From the explicit form of $G$ given in
(\ref{green}) it follows that for $A\subset  (-1,1) \times (-1,1),$  $I_G(\mu)<I_{-\pi^{-1}(\ln |x-y|)}(\mu)
+1.$

Thus, with $\mu_0$ an equilibrium measure of a closed $A\subset  (-1,1) \times (-1,1),$

$$I_M(\mu_0)\leq C( I_{-\ln |x-y|}(\mu_0)+1) = C(1-\ln\opCap (A)).$$

This implies that $$\P\{\mbox{hitting}\;\, A\;\, \mbox{before escape
  from}\;\, (-2,2) \times (-2,2)\} \geq \frac c{1-\ln\opCap (A)}\geq
c\opCap (A).$$

Since for Borel $B\subset\R^2,$ 
\begin{equation}\label{c}
\opCap (B) =\sup_{K } \opCap (K)
\end{equation}
where  $\sup$ runs over
compact subsets of $B,$ and probability of hitting $B$ before escape is
bounded below by probability of hitting $A$ for $A\subset B,$
 it is enough to estimate the logarithmic capacity of $B$  by $c\rho.$ 

Note that for subspaces $V\subset \R^2,$
\begin{equation}\label{d}
|Proj_V A| =\sup_K|Proj_V K|
\end{equation}
where  $\sup$ runs over
compact subsets of $B$ and $|\cdot|$ stands for the Lebesgue measure
in $V.$   To prove the nontrivial inequality in
(\ref{d}) observe that by the measurable section Theorem one can find
a measurable function $f:Proj_V A\to A$ such that $Proj_V f(y) =y.$
Then by Luzin's theorem, for any $\epsilon >0,$  $f$ is continuous on
a compact $C\subset Proj_VA$   of measure
at least $|Proj_V A|-\epsilon$, and thus $f(C)\subset A$ is a compact
with the desired measure of projection.

We now  use that for compact sets capacity coincides with the
transfinite diameter: \begin{equation}\label{cap}\opCap K = \lim_{n\to\infty}
\max_{z_1,...,z_n\in K}\bigg( \displaystyle\prod_{1\leq j<k\leq
  n}{|z_j-z_k|}\bigg)^{\frac 2{n(n-1)}}\end{equation}
Clearly,  for any compact $K\subset A$ the RHS of (\ref{cap})  is minorated
by the same quantity with $K$ replaced by the $Proj_VK.$

 It remains to note that for any
Borel $D\subset [0,1]$ of Lebesgue measure $\rho$ and any $b<\rho/n,$
there exist $z_1,...,z_n\in D$ that belong to an arithmetic
progression with step $b$, or equivalently with $|z_i-z_j|=k/b,$ some
$k,$ see e.g. \cite{j99}. 
Estimating the RHS of (\ref{cap}) for such $z_1,...,z_n$ leads to the claim.

 
\end{proof}
We can now move to the main Lemma of this section.
\begin{lemma}
\label{lemma-2}
Let $\epsilon>0$, $\delta>0$.
Let $\alpha \in \R \setminus \Q$ and let $\frac {p} {q} \in \Q$ be a
continued fraction approximant, and $q'$ the denominator of the
previous approximant.
Let $\psi$ be a subharmonic function on $|\im\, z|<\epsilon$ satisfying
$\psi(z) \leq 1$,
$\underline \psi=\inf_{|t|<\epsilon} \sup_{x \in \R/\Z} \psi(x+i t) \geq 0$,
and let $T$ be the set
of all $t \in (-\epsilon,\epsilon)$ such that
\be
\inf_{x \in \R/\Z}
\sup_{0 \leq k \leq q+q'-1} \psi(x+k \alpha+it) \leq \underline \psi-\delta.
\ee
Then $|T| \leq c_q$, where $c_q=c_q(\epsilon,\delta)>0$ satisfies
$\lim_{q \to \infty} c_q=0$.

\end{lemma}

\begin{proof}

Let $M$ be maximal with $\frac {M} {q}+\frac {1} {2q}<\epsilon$.
We say that some $j \in [-M,M]$
is $\rho$-bad if $|T \cap (\frac {j} {q}-\frac {1}
{2 q},\frac {j} {q}+\frac {1} {2 q})|>\frac {\rho} {q}$.

Let $B$ be the set of all $z$ with $|\im\, z|<\epsilon$ such that $\psi(z)
\leq \underline \psi-\delta$.  Notice that if $t \in T$ then there exists
$y \in \R/\Z$ such that $y+k \alpha+i t \in B$, $0 \leq k \leq q+q'-1$. 
Notice that $\{y+k \alpha+i t\}$ is $\frac {1} {q}$ dense in the circle $\im\,z=t$.

Let us consider any point of the form $x_0+i \frac {j} {q}$ where $x_0 \in
\R/\Z$ and $j$ is $\rho$-bad.  Let us start Brownian motion at $x_0+i\frac{j}{q}$, and
run it until it escapes $|\im\, z-\frac {j} {q}|<\frac {1} {q}$.  Then the
probability that the Brownian motion does not hit $B$ before escaping is
at most $e^{-\kappa}$ for some $\kappa=\kappa(\rho)>0$.  Indeed, this
probability is at most
that of escaping the square $(x_0-\frac {1} {2q},x_0+\frac
{1} {2q}) \times (\frac {j} {2q}-\frac {1} {2q},
\frac {j} {q}+\frac {1} {q})$ without hitting $B$.  One easily sees that
$B$ is a big obstacle for the Brownian motion in this rectangle by noticing
that the projection of $B
\cap (x_0-\frac {1} {2 q},x_0+\frac {1} {2 q}) \times (\frac {j} {q}-
\frac {1} {2 q},\frac {j} {q}+\frac {1} {2 q})$
on the second coordinate has measure at least $\frac {\rho} {q}$.
Therefore, by Lemma~\ref{lemma-1} $\kappa(\rho)\geq-\ln(1-c\rho)\geq c\rho$.

Assume that the number of $\rho$-bad $j$'s is either $2l$ or $2 l-1$.  Let $j_0$
be such that there are at least $l-1$ $\rho$-bad $j$'s bigger than
$j_0$ and at least $l-1$ $\rho$-bad $j$'s smaller than $j_0$.

Fix $x_0 \in \R/\Z$ such that $\psi(x_0+i \frac {j_0} {q})
\geq \underline \psi$.
Let us start Brownian motion from $x_0+i \frac {j_0}
{q}$, and run it until it escapes $|\im\, z|<\epsilon$.  Then
\be
\underline \psi \leq
\psi(x_0+i \frac {j_0} {q}) \leq p_0+(1-p_0) (\underline \psi-\delta),
\ee
where $p_0$ is the probability that Brownian motion escapes without hitting
any point in $B$.
Since $\underline \psi \geq 0$, we have
\be
\frac {p_0} {1-p_0} \geq \delta.
\ee

When the Brownian motion escapes, it has to go at least through $l-1$ layers of $\rho$-bad $j$'s, therefore
$p_0 \leq e^{-(l-1) \kappa}$, implying 
$l-1 \leq \frac{\ln((\delta+1)/\delta)}{\kappa} \leq \frac{C\ln(1+\frac
  1{\delta})}{\rho}.$
Since one has at most $2l$ $\rho$-bad $j$'s and $\epsilon-(\frac{M}{q}+\frac{1}{2q})<\frac{1}{q}$, one finds 
that $|T| \leq \frac {(2 l+2)} {q}+2 \rho \epsilon \leq C \frac {-\ln \delta} {\rho q}+2 \epsilon \rho$. Optimizing for
$\rho$ gives $|T| \leq C \epsilon^{1/2} q^{-1/2} (\ln (\frac
  1{\delta}))^{1/2}$.
\end{proof}

\section{A criterion for domination}

We will need a few well known properties of dominated cocycles.  The
discussion below is parallel to the $\SL(2,\R)$ case\footnote {In this case,
$1$-domination is the same as uniform hyperbolicity.} carried out in
detail in Section 2.1 of \cite {Av2}, so we omit the proofs.

The set of $k$-dimensional subspaces of $\C^d$ is a compact Grassmannian manifold with a holomorphic
structure (cf. Appendix~\ref{app}) and will be denoted by $G(k,d)$.
A $k$-conefield is an open set $U \subset \R/\Z \times G(k,d)$,
such that for every $x \in \R/\Z$ there exists $w \in G(k,d)$ and
$w' \in G(d-k,d)$ such that $(x,w) \in U$ and
$(x,\tilde w) \notin \overline U$ whenever
$\tilde w$ is not transverse to $w'$.  If $(\alpha,A)$ is $k$-dominated,
then it is easy to construct a $k$-conefield $U$ such that for every $(x,w)
\in \overline U$, then $w$ is transverse to the kernel of $A(x)$ and
$(x+\alpha,A(x) \cdot w) \in U$.
Conversely, $k$-domination can be detected by a {\it conefield
criterion}: there exist $n \geq 1$ and a $k$-conefield $U$
such that for every $(x,w) \in \overline
U$, 
$(x+n \alpha,A_n(x) \cdot w) \in U$.
The conefield criterion implies that $k$-domination holds
through an open set of $(\alpha,A) \in \R \times C^0(\R/\Z,\Ld)$.  

The dominated splitting for a $k$-dominated cocycle will be typically denoted by
$\C^d=u(x) \oplus s(x)$, $u(x)\in G(k,d),\,s(x)\in G(d-k,d)$.

In the particular case where $A$ admits a holomorphic extension through
$|\im\, z|<\epsilon_0$, we see that
there exists $0<\epsilon<\epsilon_0$ such that
$(\alpha,A(\cdot+i t))$ is $k$-dominated for $|t|<\epsilon$ and with
invariant sections of the form $u(\cdot+i t)$ and $s(\cdot+i t)$, with $u$
and $s$ holomorphic through $|\im\, z|<\epsilon$ 
(cf. Theorem~\ref{holomorphic}).

Before we can state our criterion we need the following lemma.
If for a matrix $B$ $\sigma_k(B)>\sigma_{k+1}(B)$ then we denote with $E_k^+(B)\in G(k,d)$ the $k$-dimensional subspace of $\C^d$ associated with the first $k$ singular values.
Moreover, $P_{E_k^+(B)}$ denotes the orthogonal projection on that subspace.

\begin{lemma}

Let $0<\rho \leq \frac {1} {4}$ be such
that $A,B$ satisfy $\sigma_2(A) \leq \rho^2 \sigma_1(A)$,
$\sigma_2(B) \leq \rho^2 \sigma_1(B)$ and $\sigma_1(B A) \geq
4 \rho \sigma_1(B) \sigma_1(A)$.  If $w \in \P \C^d$ satisfies
$\|P_{E^+_1(A)}|w\| \geq \rho$ then
$\|P_{E^+_1(B)}|(A \cdot w)\| \geq 2 \rho$.

\end{lemma}

\begin{proof}

Let $\gamma=\|P_{E^+_1(B)}|A \cdot E_1^+(A)\|$.

Let $v \in E^+_1(BA)$ be a unit vector, and let $y=P_{E^+_1(A)} \cdot v$,
$z=v-y$.  Then
$\sigma_1(BA)=\|BA \cdot v\| \leq \|BA \cdot y\|+\|BA \cdot z\|$.
Clearly $\|BA \cdot z\| \leq
\sigma_1(B) \sigma_2(A) $ and $\|BA \cdot y\| \leq
\gamma \sigma_1(A) \sigma_1(B)+\sigma_1(A) \sigma_2(B)$.  It follows that
$\gamma \geq 4 \rho-2 \rho^2\geq 3.5 \rho$, as $\rho\leq\frac14$.

Let now $w \in \C^d$ be a unit vector  such that $\|P_{E^+_1(A)} w\| \geq
\rho$, write $w=u+x$ with $u=P_{E_1^+(A)}(w)$, 
then $\|u\|\geq \rho$ and hence
\begin{align}
\frac{\|P_{E^+_1(B)} A w\|}{\|Aw\|} &\geq
 \frac{\gamma \|Au\|-\|Ax\|}{\|Au\|+\|Ax\|}\geq
\frac{\gamma \sigma_1(A) \|u\|-\sigma_2(A)}{\sigma_1(A)\|u\|+\sigma_2(A)}\geq \notag \\
&\geq \frac{\gamma-\frac{\rho^2}{\|u\|}}{1+\frac{\rho^2}{\|u\|}}\geq
\frac{\gamma-\rho}{1+\rho}\geq\frac45\,\cdot\,\frac52\rho
= 2\rho\;. 
\end{align}
Thus $\|P_{E^+_1(B)}|A \cdot w\| \geq 2 \rho$.
\end{proof}

Now we can formulate a criterion for domination.

\begin{lemma}
\label{lemma-domination}
Assume that there exists $n \in \N$ and $0<\rho \leq \frac {1} {4}$ such
that for every $x \in \R/\Z$, $\sigma_2(A_n(x)) \leq \rho^2
\sigma_1(A_n(x))$, $\sigma_2(A_n(x+n \alpha)) \leq \rho^2 \sigma_1(A_n(x+n
\alpha))$ and $\sigma_1(A_{2 n}(x)) \geq 4 \rho \sigma_1(A_n(x+n \alpha))
\sigma_1(A_n(x))$.  Then the cocycle
$(\alpha,A)$ is $1$-dominated.

\end{lemma}

\begin{proof}
The set $U=\{(x,w)\,:\,\|P_{E^+_1(A_n(x))}|w\|>\rho\}$ is a conefield and satisfies the
conefield criterion for domination.
\end{proof}



\section{Continuity of the Lyapunov exponents}

Recall that $L_j(\alpha,A)$ denotes the $j$-th Lyapunov exponent of the cocycle
$(\alpha,A)$, $L^k=\sum_{j=1}^k L_j$, $L^k(\alpha,A)=L_1(\alpha,\Lambda^k A)$.

From now on we consider cocycles $(\alpha,A)$ with $A$ analytic.

\begin{lemma}
\label{lemma-dominated}
Let $\alpha \in \R \setminus \Q$.
Assume that $L_k(\alpha,A)>L_{k+1}(\alpha,A)$.
Then there exists $\epsilon>0$ such that for almost every
$|t|<\epsilon$ the cocycle $x \mapsto A(x+i t)$ is
$k$-dominated.

\end{lemma}

\begin{proof}

Taking exterior products, we reduce to the case $k=1$.  Let
$L_1=L_1(\alpha,A)$ and $L_2=\max \{L_2(\alpha,A),L_1(\alpha,A)-1\}$,
$L^1=L_1$, $L^2=L_1+L_2$.

Fix $0<\kappa<\frac {L_1-L_2} {24}$.
By unique ergodicity\footnote{ see a more detailed argument in the
  proof of  Lemma~\ref{lem-approx}.}, there exists $n_0 \in \N$ such
that
\be
\sup_{x \in \R/\Z}
\frac {1} {n} \ln \|\Lambda^j A_n(x)\| \leq L^j+\kappa,
\quad 1 \leq j \leq 2,
\ee
holds for $n \geq n_0$.
Fix $\epsilon_0>0$ sufficiently small so that
\be \label{lem5-bound-above}
\sup_{|\im\, z|<\epsilon}
\frac {1} {n} \ln \|\Lambda^j A_n(z)\| \leq L^j+2\kappa, \quad 1
\leq j \leq 2,
\ee
holds for $n_0 \leq n \leq 2 n_0-1$, and hence (by subadditivity) for all $n
\geq n_0$.

The function $t \mapsto L^j(\alpha,A(\cdot+i t))$ is convex, and hence
continuous.  Take $0<\epsilon<\epsilon_0$ such that for
$|t|<\epsilon,$ $L_1(\alpha,A(\cdot+i
t)) \geq L_1-\kappa$.  In particular, for $|t|<\epsilon$ we have
\be \label{lem5-bound-below}
\sup_{x \in \R/\Z} \frac {1} {n} \ln \|A_n(x+it)\| \geq L_1-\kappa.
\ee

Fix a continued fraction  approximant $\frac {p} {q}$ of $\alpha$ and let $q'$ be the
denominator of the previous approximant.
For any $n \geq n_0$, let
$\phi_n(z)=\frac {1} {n} \ln \|A_n(z)\|$ which is subharmonic 
in $|\im\, z|<\epsilon$.  Notice that
\be \label{lem5-est-phi}
\sup_{0 \leq k \leq q+q'-1} \phi_{n+q+q'}(z-q \alpha+k \alpha) \leq \frac {n}
{n+q+q'} \phi_n(z)+\frac {q+q'} {n+q+q'} \sup_{|\im\, z|<\epsilon} \ln \|A(z)\|.
\ee

Let $T_n$ be the set of all $|t|<\epsilon$ such that there exists
$x \in \R/\Z$ with $\phi_n(x+i t) \leq L_1-3 \kappa$, and let
$T_{n,q}$ be the set of all
$|t|<\epsilon$ such that there exists $x \in \R/\Z$ with
$\phi_n(x+i t+k \alpha) \leq L_1-2 \kappa$ for all $k=0,\ldots,q+q'-1$.  By \eqref {lem5-est-phi}, there
exists $n(q) \in \N$ such that for $n \geq n(q)$ we have $T_n \subset
T_{n+q+q',q}$.

By Lemma~\ref{lemma-2} (applied to the function
$\psi=\frac {\phi_n-(L_1-\kappa)} {3 \kappa}$ and $\delta=\frac
{1} {3}$),
for $n \geq n_0$ we have $|T_{n,q}| \leq c_q$, with $\lim c_q=0$.  Thus for
$n \geq \max \{n_0,n(q)\}$ we have $|T_n| \leq c_q$ as well.  It follows
that $\lim |T_n|=0$.

In particular, for almost every $|t|<\epsilon$, there exist arbitrarily
large $n$ with $t \notin T_n \cup
T_{2 n}$.  Fix such $t$ and $n$.  Then for every $x
\in \R/\Z$, letting $W_1=A_n(x+i t)$ and $W_2=A_n(x+i t+n \alpha)$, so that
$W_2 W_1=A_{2 n}(x+i t)$, we have by \eqref{lem5-bound-above} and \eqref{lem5-bound-below}
\be
\sigma_1(W_2 W_1) \geq e^{-8 \kappa n} \sigma_1(W_2)
\sigma_1(W_1),
\ee
as well as
\be
\sigma_2(W_j)=\frac{\|\Lambda^2 W_j\|}{\|W_j\|^2} \sigma_1(W_j) \leq e^{(L_2-L_1+8 \kappa) n}
\sigma_1(W_j), \quad 1 \leq j \leq 2.
\ee
For large $n$ (such that $e^{(L_2-L_1+24 \kappa) n} \leq
\frac {1} {16}$), we can apply Lemma~\ref{lemma-domination} with $\rho=\frac 14e^{-8 \kappa
n}$, to conclude that $(\alpha,A(\cdot+i t))$ is $1$-dominated.
Note that the whole argument also works if $A$ is not invertible and even if
$L_2(\alpha,A)=-\infty$. By taking exterior products
the case where $L_k(\alpha,A)$ is finite but $L_{k+1}(\alpha,A)=-\infty$ is also covered. 
\end{proof}


Now we are ready to prove Theorem~\ref{thm-cont}.


\renewcommand{\proofname}{Proof of Theorem~\ref{thm-cont}}

\begin{proof}

Consider first the case $L_k(\alpha',A')=-\infty$. As $L^k=L_k+L^{k-1}$, this means 
$L^k(\alpha',A')=-\infty$. By upper-semi continuity, $L^k$ is continuous at $(\alpha',A')$.
As $kL_k\leq L^k$ we obtain for $(\alpha_n,A_n)\to(\alpha',A')$ that
$L_k(\alpha_n,A_n)\to-\infty$ as well, showing continuity. 

Let $L_k(\alpha',A')>-\infty$.
By the definition of the inductive topology, it is enough to consider the
restriction to $C^\omega_{\epsilon_0}(\R/\Z,\Ld)$ for arbitrary
$\epsilon_0>0$. Let $\alpha'\in \R\backslash\Q.$

If $L_k(\alpha',A')>L_{k+1}(\alpha',A')$, we choose $\epsilon>0$ small
such that
$(\alpha',A'(\cdot+i t))$ is $k$-dominated for $t=\pm \epsilon,\pm 2
\epsilon$.  Then $(\alpha,A) \mapsto
L^k(\alpha,A(\cdot+i t))$ is continuous in a neighborhood of $(\alpha',A')$
for $t=\pm \epsilon,\pm 2\epsilon$.
Let $s_{(\alpha,A)}(a,b)=(b-a)^{-1}[L^k(\alpha,A(\cdot+ib))-L^k(\alpha,A(\cdot+ia))]$ be the slope of
the secant of the function $t\mapsto L^k(\alpha,A(\cdot+it))$ from $a$ to $b$.
By convexity, for $|t|<\epsilon$ one finds
\be
s_{(\alpha,A)}(-\epsilon,-2\epsilon) \leq s_{(\alpha,A)}(0,t) \leq s_{(\alpha,A)}(\epsilon,2\epsilon)\;.
\ee
Since $(\alpha,A)\mapsto s_{(\alpha,A)}(\pm \epsilon, \pm 2\epsilon)$ are continuous at $(\alpha',A')$, we find a neighborhood
$\UU$ of $(\alpha',A')$ and a uniform constant $C$, such that for 
$(\alpha,A)\in\UU$ and $|t|<\epsilon$,
$|L^k(\alpha,A(\cdot+i t))-L^k(\alpha,A)| \leq C |t|$.  Considering a
sequence $t_n \to 0$ for which $(\alpha',A'(\cdot+i t_n))$ is $k$-dominated,
and hence $(\alpha,A) \mapsto L^k(\alpha,A(\cdot+i t))$ is continuous on a
(possibly decreasing with $n$) neighborhood of $(\alpha',A')$, we get that
$L^k$ is continuous at $(\alpha',A')$.

Assume now that $L_j(\alpha',A')=L_k(\alpha',A')>-\infty$ for $j$
in a maximal interval $[a,b]$ containing $k$.
Then $L^{a-1}$ and $L^b$ are continuous at
$(\alpha',A')$.\footnote{
  If $b=d$ it follows as well since $L^d =\int\ln |\det A | dx,$ which
  is continuous on $C^\omega(\R/\Z,\Ld)$, see e.g. \cite{jm2}. 
  For $a=1$ we just define $L^0=0$ so that $L_1=L^1-L^0$.}
Since $L^a$ and $L^{b-1}$ are
upper-semicontinuous, $L_a$ is
upper-semicontinuous at $(\alpha',A')$ and $L_b$ is lower-semicontinuous at
$(\alpha',A')$.  Since $L_a \geq L_j \geq L_b$ for $a \leq j \leq b$
and $L_a(\alpha',A')=L_b(\alpha',A')$ by hypothesis,
we conclude that
$L_j$ is continuous at $(\alpha',A')$
for $a \leq j \leq b$.  The result follows.
\end{proof}

\renewcommand{\proofname}{Proof.}


\section{Regularity and approximation through rationals}

Recall that
\be
\omega^k=\lim_{\epsilon \to 0+} \frac {1} {2 \pi \epsilon}
(L^k(\alpha,A(\cdot+i \epsilon))-L^k(\alpha,A)).
\ee
We let the frequency $\alpha$ be irrational from now on.
Furthermore we assume that $A$ extends to a complex analytic function
in a neighborhood of $|\im\,z|\leq \delta$.

Recall that $(\alpha,A)$ is $k$-regular if $t \mapsto L^k(\alpha,A(\cdot+i
t))$ is an affine function for $|t|<\epsilon$.

Let
$$
\R\setminus \Q \ni \alpha=\lim_{n\to\infty}
\frac{p_n}{q_n}\qtx{with} p_n,q_n \in \Z_+\;,\;\;(p_n,q_n)=1\;
$$
and define for $z=x+it$ and $p/q\in\Q$
$$ L^k(p/q,A,x):=\lim_{n\to\infty} 1/n \ln \|\Lambda^k A_n(p/q,x)\| $$
Clearly, it exists for all $x\in \T$ and we have $ L^k(p/q,A,x)=\frac{1}{q}\ln
\rho (\Lambda^k A_q(p/q,x))$ where $\rho(A_*)$ is the spectral radius of $A_* \in
\LL(\C^d,\C^d)$. By definition, $$L^k(p/q,A)=\int_{\R/\Z}L^k(p/q,A,x)dx.$$

We start with the following crucial lemma.

\begin{lemma} \label{lem-approx}
If $L^k(\alpha,A)>-\infty$ then one has uniformly for small $t$ and all $x$ that
\be\label{eq-est-rat}
L^k\left(\frac{p_n}{q_n},A(\cdot +it),x\right)\;\leq\;L^k(\alpha,A(\cdot+it))\;+\;o(1)\;.
\ee
More precisely, the estimate being uniform means that for some $\delta >0$,
\be
\limsup_{n\to\infty} \;\sup_{x\in\R/\Z}\; \sup_{|t|\leq \delta} 
\left[L^k\left(\frac{p_n}{q_n},A(\cdot +it),x\right)-L^k(\alpha,A(\cdot+it))\right] \leq 0
\ee
If $L^k(\alpha,A)=-\infty$ then $L^k\left(\frac{p_n}{q_n},A(\cdot +it),x\right)$ converges to $-\infty$ for $n\to \infty$, uniformly
for all $x$.
\end{lemma}

\begin{proof}
 By taking exterior products we may just consider the case $k=1$.
Let $\Phi(t)=L_1(\alpha,A(\cdot+it))$. We first assume $\Phi(0)>-\infty$. 
By Theorem~\ref{thm-iii}, $\Phi(t)$ is piecewise affine (note that the proof of Theorem~\ref{thm-iii} depends on Lemma~\ref{lemma-dominated}
and \ref{lem-omega-int} which do not depend on this lemma). Take $\delta>0$ such that $\Phi(t)$ is affine on $[-\delta,0]$ and $[0,\delta]$ with a possible
corner at $0$.
Choose $n$ such that $\frac 1n\int_{\R/\Z} \ln \|A_n(\alpha,x+it) \|d x < \Phi(t)+\epsilon$ 
for $t\in\{-\delta,0,\delta\}$.
By unique ergodicity we get uniform upper bounds in the
ergodic theorem applied to $\ln \|A_n(\alpha,x+it) \|,$ so there
exists $j$ such that for all $x$ and $t\in\{-\delta,0,\delta\}$
$$
\frac 1{jn}
\displaystyle\sum_{k=0}^{j-1} \ln \|A_n(\alpha,x+kn\alpha+it)\|<\frac 1n\left(\int_{\R/\Z} \ln \|A_n(\alpha,x+it) \|d x +\epsilon\right)<\Phi(t)
+2\epsilon.
$$ 
Thus  for $m=jn$ we have by subadditivity,
$\frac1m \ln \|A_m(\alpha,x+it)\|-\Phi(t)<2\epsilon$  for any $x$ and $t\in\{-\delta,0,\delta\}$. 
By continuity and compactness we find $N>0$ such that for $n>N$,
$\frac1m \ln\|A_m(p_n/q_n,x+it)\|-\Phi(t)<3\epsilon$ for all $x$ and $t\in\{-\delta,0,\delta\}$. 
The left hand side is subharmonic for $t\in(-\delta,0)$ and $t\in(0,\delta)$. Therefore, by the maximum principle,
the last estimate holds for all $|t|\leq \delta$.
By subadditivity, for $K$ large enough and any $r=0,\ldots,m-1$ one has
uniformly for $|t|\leq\delta$
$$
\frac{1}{Km+r}\ln\|A_{Km+r}(p_n/q_n,x)\|\leq \frac{1}{Km+r}(Km(\Phi(t)+3\epsilon)+Cr) < \Phi(t)+4\epsilon\;.
$$
This proves the claim. If $\Phi(0)=-\infty$ then by continuity and convexity this happens for all $t$ where the 
holomorphic extension $A(x+it)$ is defined and we can change $\Phi(t)$ to $-1/\epsilon$ in the estimates.
\end{proof}

If the cocycle is $k$-regular, then one can approximate the Lyapunov exponent by using 
rational frequencies and any phase $x$. This is the main result in this section.

\begin{thm} \label{thm-approx}
Assume $L^k(\alpha,A)>-\infty$ and that $(\alpha,A)$ is $k$-regular.
Then one has uniformly for small $t$ and all $x\in\R/\Z$ 
\be \label {ii}
L^k\left(\frac{p_n}{q_n},A(\cdot+it),x\right)\;=\;
L^k\left(\alpha,A(\cdot+it)\right)\,+\, o(1)
\ee
\end{thm}

\begin{proof}
Again by using exterior products it is enough to consider $k=1$.
Assume that $A$ admits a holomorphic extension to $|\im (x)|<\delta_1$ bounded
by $C>0$ and that $\Phi(t)=L_1(\alpha,A(\cdot+i t))$
is affine for $|t| \leq 
\delta_0<\delta_1$.  Up to multiplying $A$ by a sufficiently large
constant, we
may also assume that $\Phi(t)>1$ for $|t| \leq \delta_0$.  We are going to
show that
\begin{equation} \label {geq}
\frac 1{q_n}\ln\rho(A_{q_n}(p_n/q_n,\cdot +i t)) \geq \Phi(t)+o(1), \quad |t| \leq
\delta_0/2,
\end{equation}

This concludes, since (\ref {eq-est-rat}) can be rewritten as
$\frac 1{q_n}\ln\rho(A_{q_n}(p_n/q_n,\cdot +i t)) \leq \Phi(t)+o(1)$, so (\ref {geq})
implies $\frac 1{q_n}\ln\rho(A_{q_n}(p_n/q_n,\cdot +i t) )=\Phi(t)+o(1)$, which is just
(\ref {ii}) for $k=1$.

It is easy to see that there exists $c_d>0$
such that for any $A_* \in \LL(\C^d,\C^d)$
there exists $1 \leq k \leq d$ such that $|\tr A^k_*|^{1/k}
\geq c_d \rho(A_*)$.\footnote {Indeed,
by homogeneity and compactness, this inequality holds with
$c_d=\min \max_{1 \leq k \leq d} |\sum_{j=1}^d \lambda_j^k|$, where the
minimum runs over all sequences $\lambda_j \in \overline \D$, $1 \leq j \leq
d$, such that $\max_j |\lambda_j|=1$, and we just
have to check that $c_d>0$.  But if $c_d=0$ then there exists $\lambda_j \in
\C$, $1 \leq j \leq d$, not all zero, such that $\sum_j \lambda_j^k=0$ for
$1 \leq k \leq d$.  Let $J \subset \{1,...,d\}$ be the set of all $j$ such
that $\lambda_j \neq 0$.  Then letting $p(z)=\prod_{j \in J} (z-\lambda_j)$ we
have $p(\lambda_j)=0$ for each $j \in J$, while $\frac {1} {\# J}
\sum_{j \in J} p(\lambda_j)=p(0) \neq 0$
(since the contributions corresponding to each non-constant monomial
add up to zero), contradiction.}
Let $1 \leq k_n \leq d$ and $x \in \R/\Z$ be such that   
$|\tr A_{q_n}(p_n/q_n,x)^{k_n}|^{1/k_n}$ is maximal.  Let
$$
\phi_n(t)=\max_{x \in \R/\Z}
\frac {1} {k_n q_n} \ln |\tr A_{q_n}(p_n/q_n ,x+it)^{k_n}|.
$$
Then
$\phi_n(0) \geq L_1(p_n/q_n,A)+\frac
{1} {q_n} \ln c_d$, and using that $L_1(p_n/q_n,A)=L_1(\alpha,A)+o(1)$
(Theorem \ref {thm-cont}), we get
$\phi_n(0) \geq \Phi(0)+o(1)$.

On the other hand, by Lemma~\ref{lem-approx} we have $\phi_n(t) \leq \Phi(t)+o(1)$,
$|t| \leq \delta_0$.  Since $\phi_n(t)$ is clearly
a convex function of $t$, and $\Phi(t)$ is affine for $|t| \leq \delta_0$,
it follows that $\phi_n(t)=\Phi(t)+o(1)$ for $|t| \leq \delta_0$.

Write $\tr A_{q_n}(x)^{k_n}=\sum_{j \in \Z} a_{j,n} e^{2 \pi i j q_n x}$.
Then $|a_{j,n}| \leq d C^{q_n} e^{-2 \pi |j| q_n \delta_1}$.  Thus we can  
choose $m_0>0$ such that $\sum_{|j|>m_0} |a_{j,n}| e^{2 \pi| j|\ q_n \delta_0}
\leq 1$ for every $n$.  It follows that
$$
\phi_n(t)=\max_{|j| \leq m_0} \left ( \frac
{1} {k_n q_n} \ln |a_{j,n}|-\frac {2 \pi j} {k_n} t \right )+o(1), \quad |t|
\leq \delta_0,
$$
and since $\phi_n(t)=\Phi(t)+o(1)$ with $\Phi$ affine, we see that there
exist $|j_n| \leq m_0$ such that the slope of $\Phi$ is $\frac {-2\pi j_n} {k_n}$
and we have
\be
\phi_n(t)=\frac {1} {k_n q_n} \ln |a_{j_n,n}|-\frac {2 \pi j_n} {k_n} t+o(1),
\ee
while for each $|t| \leq \delta_0/2$ and $|j| \leq m_0$ we have
\be
\frac {1} {k_n q_n} \ln |a_{j,n}|-\frac {2 \pi j} {k_n} t\leq
\frac {1} {k_n q_n} \ln |a_{j_n,n}|-\frac {2 \pi j_n} {k_n} t-
\frac {\delta_0 \pi |j-j_n|} {k_n}+o(1).
\ee
It follows that
$$
\frac {\tr A_{q_n}(z,p_n/q_n)^{k_n}} {a_{j,n} e^{2 \pi i j q_n z}}=1+o(1),
\quad z=x+it,\quad |t| \leq \delta_0/2,
$$
so that
$$
\frac {1} {k_n q_n} \ln |\tr A_{q_n}(z,p_n/q_n)^{k_n}| \geq
\Phi(t)+o(1), \quad |t| \leq \delta_0/2.
$$
Thus $\frac {1} {q_n}
\ln \rho(A_{q_n}(p_n/q_n ,z)) \geq \Phi(t)+o(1)$ for $|t| \leq \delta_0/2,$ as
desired.
\end{proof}

\section{Holomorphic dependence and convergence}

In this section we will finally prove the main theorems.
In order to obtain the equivalence of regularity and domination as stated in Theorem~\ref{thm-i} we will argue
with approximation of the unstable and stable directions by rational frequencies and convergence of
holomorphic functions.
As before, $G(k,d)$ denotes the Grassmannian of $k$-dimensional subspaces of
$\C^d$. As described in Appendix~\ref{app} this is a holomorphic manifold.
An important fact is the holomorphic dependence of dominated splittings:

\begin{thm}\label{holomorphic}
Let $(\alpha,A(\cdot+it))$ be $k$-dominated for $t\in(t_-,t_+)$ and let $u(x+it)\oplus 
s(x+it)$ be the corresponding dominated splitting.
Then $z\mapsto u(z)\in G(k,d)$ and $z\mapsto s(z)\in G(d-k,d)$ are holomorphic 
for $z=x+it, \,t\in(t_-,t_+)$.
\end{thm}

We first consider just the more unstable directions in the dominated splitting 
and start with an analogue to Lemma~2.1 in \cite{Av2} showing holomorphic dependence.
This means in the considered splitting $\C^d=u(x)\oplus s(x)$ we assume that for some $n$, 
any $x$ and any unit vectors $w\in u(x),\,v\in s(x)$ we have
$\|A_n(x) w \|> \|A_n(x)v\|$.
As a corollary we will obtain Theorem~\ref{holomorphic} for rational frequencies. The holomorphic
dependence of $s(z)$ for irrational frequencies will be concluded in the proof of 
Theorem~\ref{thm-i}.\footnote{If $A(z)$ is always invertible, then the holomorphic dependence
of $s(z)$ follows directly from Lemma~\ref{lem-holomorph} by considering the inverse cocycle, but the singular case 
requires approximation by rational frequencies.}

\begin{lemma}\label{lem-holomorph}
Let $\DD\OO_k(\alpha,\C^d)$ denote the set of  $k-$dominated analytic cocycles on $\C^d$ with frequency $\alpha$.
For any $x\in\R/\Z$ the map $A\mapsto u_A(x)$ is a holomorphic function 
of $A\in \DD\OO_k(\alpha,\C^d)$. Here, $u_A(x)$ denotes the corresponding unstable subspace.
\end{lemma}

In particular, an immediate corollary is
\begin{cor}
\begin{enumerate}[(i)]
\item the unstable subspace 
$u(x+it)\in G(k,d)$ depends holomorphically on $x+it$;
\item if $\alpha\in\Q$ is rational, then the stable subspace $s(x+it)$ depends holomorphically on $x+it$.
\end{enumerate}

\end{cor}
\begin{proof}
 Holomorphic dependence of $u_1 \wedge u_2 \wedge \ldots \wedge u_k \, \in \P(\Lambda^k \C^d)$ implies holomorphic dependence of
the subspace spanned by $u_1,\ldots,u_k$. In fact, $G(k,d)$ can be considered as a closed 
submanifold\footnote{being precisely those elements that can be written as $v_1\wedge v_2 \wedge \ldots\wedge v_k$} 
of the projective space
$\P(\Lambda^k \C^d)$. Therefore, we may consider $\Lambda^k A$ and can assume $k=1$.
Now let $\epsilon_0$ be the infimum of the distance between $u_A(x)$ and unit vectors in $s_A(x)$.
Let $0<\epsilon<\epsilon_0/2$ and consider the conefield
$U=\{x,m\},\,m\in \P\C^d$ such that $m$ is $\epsilon$ close to
$u(x)$. Here we use the spherical metric on $\P\C^d.$ Note that $A$ acts
on the  $\P\C^d$ in a natural way.
Take $n$ large enough such that $(x+n\alpha,A_n(x)\cdot m)\in U$ for every $(x,m)\in\overline U$.
Let $\VV\subset \DD\OO_1(\alpha,\C^d)$ be the set of all $(\alpha,A')$ such that
$(x+n\alpha,A_n(x)\cdot m)\in U$ for every $(x,m)\in\overline{U}$. $\VV$ is an open neighborhood of $A$ and for $A'\in\VV$ we find
that $u_{A'}(x)$ is the limit $k\to\infty$ of $u^k_{A'}(x)=A'_{kn}(x-kn\alpha)\cdot u_A(x-kn\alpha)$. For each $k\geq 1$ this is a holomorphic function of $A'$
taking values in the hemisphere of $\P\C^d$ centered at $u_A(x)$. By Montel's Theorem, the limiting function
$A'\mapsto u_{A'}(x)$ is holomorphic.

Part (i) of the corollary follows by holomorphy in $\Delta z$ for 
$A'_{\Delta z}(z)=A(z+\Delta z)$. Then $u_{A'_{\Delta
    z}}(z)=u_A(z+\Delta z)$.

For part (ii) first note that taking $\alpha=0$ shows that the eigenvector corresponding to the largest modulus of the eigenvalues 
of a holomorphic matrix valued function $B(z)$ with a gap between the
largest and second largest eigenvalues depends holomorphically on $z$.
Using tensor products and inverses $(\Lambda^k B(z)+\epsilon \mathbf{1})^{-1}$ we find that the direct sums of generalized 
eigenspaces\footnote{The generalized eigenspace for a $d\times d$ matrix $B$ to the eigenvalue $\lambda$ is the kernel of 
$(B-\lambda)^d$.}
(corresponding to Jordan blocks)
of eigenvalues of modulus greater or smaller than a constant $c$ depend also holomorphically on $z$ in a neighborhood
where no eigenvalue has modulus $c$.
For rational $\alpha=\frac{p}{q}$, the subspace $s(z)$ is locally characterized as such a subspace, where $c$ is between the 
$k$-th and $k+1$st largest modulus of eigenvalues of $A_q(z)$.
\end{proof}

Using the analyticity of $u$ we obtain the following.

\begin{lemma}\label{lem-omega-int}
If $(\alpha,A)$ is $k$-dominated then $\omega^k$ is a constant integer in a
neighborhood of $(\alpha,A)$.
Moreover, if $\det A(x)\neq 0$ for all $x$, then $\omega^d$ is a constant integer
in a neighborhood of $(\alpha,A)$. 
\end{lemma}

\begin{proof}
It is enough to consider the case $k=1$.
As in Appendix~\ref{app}, Theorem~\ref{holomorphic-G(k,d)}~(vi) we lift $u(z)\in \P\C^d$ 
to a one-periodic, holomorphic function 
$u(z)\in\C^d\setminus{0}$. Then $A(z)u(z)=\lambda(z) u(z+\alpha)$ for a one-periodic, 
holomorphic function $\lambda(z)$. 
Note, $u(z)$ and $\lambda(z)$ also depend holomorphically on $A$.
Thus, for $z=x+it$,
$L^1(\alpha,A(\cdot+it))=\int_0^1 \ln\|A(z)u(z)\|-\ln\|u(z)\|dx=
\int_0^1 \ln|\lambda(z)|\,dx$.
A direct computation (see e.g. \cite{jm2}) shows that $\omega^1(\alpha,A)=
\frac d{d\epsilon}\big| _{\epsilon = 0} \int_0^1 \ln |\lambda
(x+i\epsilon)| dx$  is minus  the winding number of
$\lambda(x)$ around $0$, so it is an integer and locally constant.
As $L^d(\alpha,A)=\int_0^1 \ln|\det A(x)|\,dx$, one obtains the same result for $\omega^d$
by the same argument.
\end{proof}


Before proving the main theorems we need another lemma that will guarantee the convergence
of the unstable and stable directions when approaching $\alpha$ by rationals.

\begin{lemma} \label{lem-uniform}
Let $D=\{z\in \C:t_-\leq \im\,z \leq t_+\}$ and let $u:D\to G(k,d),\,s:D\to G(d-k,d)$ be holomorphic functions on
the interior $\mathring{D}$ and continuous on $D$. 
Assume that $u$ is transverse to $s$ at every point and the angle
is minorated by $\epsilon$ at the boundary $\partial D$.  
Then it is minorated by $\epsilon$ in the whole strip. 
Moreover, for any compact subset $K\subset\mathring{D}$ of the open strip, $u$ and $s$ are $C$-Lipschitz
where $C$ depends only on $\epsilon$ and $K$.
\end{lemma}

\begin{proof} 
Let $P$ be the projection on $u$ along $s$, i.e. $P$ is the unique matrix
with $\ker P=s$ and $P|u=\id|u$. 
By Theorem~\ref{holomorphic-G(k,d)}~(v) we can locally 
lift the pair $(u,s)$ to a holomorphic function $B \in \GL(d,\C)$ 
where the first $k$ vectors represent $u$ and the last $d-k$ column vectors represent $s$.
Then, $P=B P_k B^{-1}$ where $P_k$ projects on the first $k$ coordinates in $\C^k$ and hence, 
$P$ is holomorphic. 
Now, $\|P\|=\sup_{\|w\|=1} \|Pw\|$ is a decreasing function\footnote{In fact, the maximum of $\|Pw\|$ occurs if $w$ lies 
in the plane with the minimal angle and is perpendicular to
$s$. Moreover, $\|P\|=\frac1{\sin(\theta)},$ e.g. \cite{gk} }
of the angle $\theta$ between $u$ and $s$, going to $\infty$ if the angle goes to zero.
However, as $P$ is holomorphic, $\|P\|=\max_{\|w\|=1} \|Pw\|$ is maximized in $D$ on the boundary $\partial D$.

For the second part, note that by Cauchy's formula, the partial derivatives of $P$ at $z_0\in\mathring{D}$ are bounded
by $C/\dist(z_0,\partial D)$ for some constant $C$ only depending on $\epsilon$.
Now, choose an orthonormal basis $w_1,...,w_k$ for $u$ at $z_0$ (they are fixed, independent of $z$) and
consider the projections $Pw_j$ as one varies the base point $z$. 
Those are Lipschitz near $z_0$ and the space they
generate (which is $u$) depends in a Lipschitz way on $z$ near $z_0$.
Using the uniform bounds of $P$ and of its derivatives on compact sets $K\subset \mathring{D}$
we obtain a Lipschitz constant $C$ only depending on $K$ and $\epsilon$.
\end{proof}

Now we are ready to prove the main theorems.

\renewcommand{\proofname}{Proof of Theorem~\ref{thm-i}}

\begin{proof}

It is enough to consider the case $k=1$.
Let $L_1(\alpha,A)>L_2(\alpha,A)$ and let $(\alpha,A)$ be 1-regular.
By Lemma~\ref{lem-omega-int} it is only left to prove that regularity 
implies domination.

We let $\frac{p_n}{q_n}$ be rational approximants with $\frac{p_n}{q_n}\to\alpha$.
By Lemma~\ref{lem-approx}
uniformly in $x$ and $|t|<\epsilon$ we have $L_1(\frac{p_n}{q_n},A(\cdot+it),x)=L_1(\alpha,A(\cdot+it))+o(1)$ and
$L^2(\frac{p_n}{q_n},A(\cdot+it),x)\leq L^2(\alpha,A(\cdot+it))+o(1)$ if $L^2(\alpha,A)>-\infty$. 
If $L^2(\alpha,A)=-\infty$ then $L^2(\frac{p_n}{q_n},A(\cdot+it),x)$ approaches $-\infty$ 
uniformly in $x$ and $|t|<\epsilon$.
Therefore, either
$L_2(\frac{p_n}{q_n},A(\cdot+it),x) \leq L_2(\alpha,A(\cdot+it))+o(1)$ or it approaches $-\infty$
and it follows that for large $n$, 
$L_2(\frac{p_n}{q_n},A(\cdot+i t),x)<L_1(\frac{p_n}{q_n},A(\cdot+i t),x)$ for every
$x \in \R/\Z$ and every $|t|<\epsilon$.  
Thus, for $n$ large, $(\frac{p_n}{q_n},A(\cdot+it))$ is $1$-dominated through a band 
$|\im\, z|=|t|<\epsilon$.

Select $t_-<0<t_+$ in this band, such that $(\alpha,A(\cdot+t_\pm))$ is $1$-dominated.
By robustness of domination, the cocycles $(\frac{p_n}{q_n},A(\cdot+t_\pm))$ are
uniformly $1$-dominated. 

By Lemma~\ref{lem-holomorph} the unstable and stable subspaces $u_n(x+it),\,s_n(x+it)$ depend holomorphically on $z=x+it$
for $t$ in neighborhood of $\{z:t_-\leq \im\, z \leq t_+\}$. By Lemma~\ref{lem-uniform} for each $n$, 
the smallest angle occurs at some point $z$ at the boundary $\im\,z=t_\pm$.
But since the  cocycles $(\frac{p_n}{q_n},A(\cdot+t_\pm))$ are
uniformly $1$-dominated, we find a uniform, non-zero lower bound for the angle between 
$u_n(x+it)$ and $s_n(x+it)$. 
Again, by Lemma~\ref{lem-uniform} the functions $u_n$ and $s_n$ are uniformly Lipschitz on compact subsets of 
$\{z: \im\, z \in (t_-,t_+)\}$.
Therefore, there is a convergent subsequence such that $u_{n_k}$ and $s_{n_k}$ converge (uniformly on compacts) to holomorphic functions
$u$ and $s$, satisfying $A(z)s(z)=s(z+\alpha)$, $A(z)u(z)=u(z+\alpha)$. 
Since $L_2<L_1$, in the limit $\frac{p_n}{q_n}\to \alpha$ 
the one-dimensional bundle $u(z)$ is associated to the top Lyapunov
exponent almost everywhere and unique ergodicity implies
domination. 
\end{proof}

Note that the limits $u(z)$ and $s(z)$ are holomorphic functions
and therefore we also proved Theorem~\ref{holomorphic}.
Next, we show the quantization of the acceleration.

\renewcommand{\proofname}{Proof of Theorem~\ref{thm-iii}}

\begin{proof} We only need to consider the case $k<d$ and $L^k>-\infty$.
Assume that $L_k(\alpha,A(\cdot+it))-L_{k+1}(\alpha,A(\cdot+it))$ is not identically zero
on $t\in[0,\epsilon]$ for any $\epsilon>0$. 
Then using Lemma~\ref{lemma-dominated} one obtains
a sequence $t_n\to 0$ where $(\alpha,A(\cdot+i t_n))$ is $k$-dominated.  
At any such $t_n$, $\omega^k(\alpha,A(\cdot+it_n))$ is an integer 
by Lemma~\ref{lem-omega-int}. By convexity of $L^k$ in $t$, $\omega^k$ must be right-continuous and constant for $t\geq 0$ 
small, hence $\omega^k\in\Z$.

Consider the case $L_k(\alpha,A(\cdot+it)) =L_{k+1}(\alpha,A(\cdot+it))>-\infty$
for $t\geq 0$ small. 
Let $[a,b]$ be the maximal interval such that there exists $\epsilon>0$ with 
$L_j(\alpha,A(\cdot+it))=L_k(\alpha,A(\cdot+i t))$ for $a \leq j \leq b$ and 
for every $t\in[0,\epsilon)$.
Let us define $L^0=0$ and $\omega^0=0$. Then, by the arguments above or 
Lemma~\ref{lem-omega-int} (in case $b=d$) we have that $\omega^{a-1}$ and $\omega^b$ are integers.
Moreover, $L^k=L^{a-1}+(L^b-L^{a-1}) \frac {k-a+1}{b-a+1}$ 
for every $0\leq t<\epsilon$. Hence, $\omega^k=\omega^{a-1}+
(\omega^{b}-\omega^{a-1})\frac {k-a+1}{b-a+1} \in \frac {1}{b-a+1} \Z$.
As $\omega^{k-1}\in \frac{1}{b-a+1}\Z$ as well\footnote{This is clear for $k\geq a+1$ and if $k=a$ then one even has 
$\omega^{k-1}\in \Z$}, one also has $\omega_k=\omega^k-\omega^{k-1} \in \frac {1}{b-a+1} \Z$.
If $A(z)\in \SL(d,\C)$ for all $z$, then $l\omega_k,\,l\omega^k \in \Z$ for an integer $1\leq l\leq d-1$.\footnote{The case $b-a+1=d$
implies $a=1$, $b=d$ and hence $\omega_k=\frac1d \omega^d$. But if $\det(A(z))=1$ then $\omega^d=0$, 
and hence all $\omega^k,\,\omega_k$ are zero.}
\end{proof}

\renewcommand{\proofname}{Proof of Theorem~\ref{thm-ii}}

\begin{proof}
As a consequence it follows immediately that $L^k(\alpha,A(\cdot+it))$ is piecewise affine.
Hence, for $t\neq 0$ small enough, $L^k$ is affine in a neighborhood of $t$. By definition, this means
that $(\alpha,A(\cdot+it))$ is $k$-regular for $t\neq 0$ small enough which proves
Theorem~\ref{thm-ii}.
\end{proof}
 
Now we have everything to prove the main Theorem.

\renewcommand{\proofname}{Proof of Theorem~\ref{main}}

\begin{proof}
By Theorem~\ref{thm-cont}, the continuity of the Lyapunov exponents, there is an open and dense
subset  $U \subset C^\omega(\R/\Z,\Ld)$ such that for $A\in U$ the
number of distinct Lyapunov exponents is locally constant. Within $U$
the set where Oseledets filtration is dominated or trivial is automatically
open. By Theorem~\ref{thm-ii} the set of cocycles that are $k$-regular
for all $k$ with $L_k>-\infty$ is dense in $U$, and by Theorem~\ref{thm-i}
all such cocycles  with not all Lyapunov exponents equal have
dominated Oseledets splitting.   

\end{proof}

\renewcommand{\proofname}{Proof}

\appendix

\section{Holomorphic quotients, submersions and lifts \label{app}}

In this appendix we want to briefly explain the holomorphic structure of the Grassmannians
$G(k,d)$ and show the existence of local holomorphic lifts to representing matrices.

Let us define the following subgroup of $\GL(d)$.
\begin{equation}
 \GL(k,d):=\left\{ \begin{pmatrix}
  A & C \\ 0 & D
 \end{pmatrix}\,:\, A\in \GL(k),\,D\in \GL(d-k),\,C\in\C^{k\times (d-k)}
 \right\}\;.
\end{equation}
Furthermore, let $\MM_k(d)$ denote the set of $d\times k$ matrices of rank $k$.

\begin{thm} \label{holomorphic-G(k,d)}
\begin{enumerate}[{\rm (i)}]
\item
The Grassmannian $G(k,d)$ can be considered as the quotient 
$$G(k,d) \,\cong\, \GL(d)\,/\,\GL(k,d) \qquad \text{(left cosets $G\, \GL(k,d)$)\;.}$$ 
\item Let $p:\GL(d)\to G(k,d)$ be the natural projection.
Then, $G(k,d)$ has a unique holomorphic structure, 
such that $p$ is a holomorphic submersion (meaning $p'$ has full possible rank everywhere).
Moreover, the left action of $\GL(d)$ is holomorphic.
\item There is a natural projection $\tilde p:\MM_k(d)\to G(k,d)$ which is also a holomorphic submersion
\item Locally, for each $G \in \GL(d)$ and $M \in \MM_k(d)$ 
there exists neighborhoods $U_G$ of $p(G)$ and $U_M$ of $\tilde p(M)$ and
holomorphic injections
$i_G: U_G \to \GL(d)$, $i_M: U_M\to \MM_k(d)$
such that $p\circ i_G = \id|U_G$
and $\tilde p \circ i_M=\id|U_M$.
\item A holomorphic function $u:D\to G(k,d)$ can be locally lifted in a small neighborhood
$U_z$ of $z\in D$
to a holomorphic function $G:U_z\to \GL(d)$ or $B:U_z\to \MM_k(d)$ such that
$p\circ G=u$ or $\tilde p \circ B=u$, respectively.
\item An analytic function $u\in C^\omega(\R/\Z,G(k,d))$ can be lifted 
to a one-periodic holomorphic function 
 $\tilde u:D_\delta \to \MM_k(d)$ such that
$\tilde p\circ \tilde u=u$, for some $\delta>0$.
Here, $D_\delta=\{\im(z)<\delta\}$,
\end{enumerate}
\end{thm}

\begin{proof}
$G(k,d)$ denotes the set of $k$-dimensional subspaces of $\C^d$.
A $k$-dimensional subspace $u\in G(k,d)$ can be represented by a basis, hence by a $d\times k$ matrix
$B(z)$ of full rank $k$ where the $k$ column vectors span $u$.
Hence, we obtain a natural projection 
$\tilde p: \MM_k(d) \to G(k,d)$. 
We also have a natural projection $\hat p:\GL(d)\to \MM_k(d)$ by simply selecting
the first $k$ column vectors. It is clear that $\hat p$ is holomorphic and we find
local holomorphic injections $i$ from small neighborhoods in $\MM_k(d)$ to $\GL(d)$ such that
$\hat p \circ i = \id$.
Therefore, statement (iii) follows from (ii) and the statement about $i_M$ in (iv) follows from
the one about $i_G$ in (iv).

Two matrices $G_1,G_2$ in $\GL(d)$ represent the same element in $G(k,d)$, 
if and only if the first $k$ column vectors span the same space.
This is equivalent to $G_1 = G_2 G$ where $G\in \GL(k,d)$.
In other words, the quotient of left-remainder classes $\GL(d)\,/\,\GL(k,d)$ is equivalent to
$G(k,d)$ and there is a natural, transitive left-action of $\GL(d)$.
We want to make $p$ a holomorphic submersion.
Therefore, consider the exponential chart $P\mapsto G\exp(P)$ around $G\in\GL(d)$.
If $p$ is a submersion, then the kernel of $p'(G)$ must precisely be given by the Lie-algebra 
$\gl(k,d)$ of $\GL(k,d)$. The Killing form $\Tr(A^*B)$ defines a natural metric on $\gl(d)$ and we can consider
the orthogonal complement $\gl(k,d)^\perp$.
Consider the map $p_G(C)=p(G\exp(C))$ for $C\in \gl(k,d)^\perp$.
For small $C$, these maps are injective.
Now, if $p$ is a holomorphic submersion, then $p_G$ must be holomorphic and the derivative at $0$
must have full rank and hence $p_G$ must
be locally invertible, i.e. $p_G$ must define a chart for small $C$. On the other hand, 
using small $C$, the maps $p_G$ for $G\in\GL(d)$ clearly define
an atlas giving $G(k,d)$ a holomorphic structure, such that $p$ is a holomorphic submersion.
Moreover, the left action of
$\GL(d)$ is also clearly holomorphic. This proves (ii).

For (iv) note that using the charts $p_G$, the maps $i_G$ defined
by $i_G(p_G(C))=G\exp(C)$ fulfill the requirement. Clearly, (v) follows from (iv).

To obtain (vi) let us consider first the case $k=1$ for simplicity. Then $G(1,d)=\P\C^d$ and
$\MM_1(d)=\C^d\setminus\{0\}$. 
It is enough to find $v\in\C^d$ such that $v$ is never orthogonal to $u(x)$, i.e.
$v^* u(x) \neq 0$ for all $x\in[0,1]$, because the canonical projection
$p_v: \{w\in\C^d: v^*w=1\} \to \P\C^d$ defines a chart and the inverse gives the desired
$1$-periodic lift $\tilde u=p_v^{-1} \circ u$.

So let $W(x)=\{w\in\C^d\,|\,w^*u(x)=0\}$, then $W(x)\cong\C^{d-1}\cong\R^{2d-2}$ defines a real, $2d-2$ dimensional 
fiber bundle over the torus $\R/\Z$ and $\MM=\bigcup_{x\in\R / \Z} \{x\}\times W(x)$ can be seen as a
real $2d-1$ dimensional submanifold of $(\R / \Z)\times \C^d$.
The map $f: \MM\to\C^d, f(x,w)=w$ is differentiable.
As $\C^d\cong\R^{2d}$ is real $2d$ dimensional, $f$ is not surjective. Take $v$ not in the image
of $f$.

For general $k$ one needs to find $V\in \MM_k(d)$ such that 
$\det(V^*u(x))\neq 0$\footnote{
the condition is independent of the representative of $u(x)$ in $\MM_k(d)$)} 
for all $x\in[0,1]$. 
Then the projection $p_V: \{W\in\MM_k(d):V^* W=1\}\to G(k,d)$ is a chart and
$\tilde u=p_V^{-1}\circ u$ will be the desired one-periodic lift. The existence of
$V$ can be obtained by similar arguments.\footnote{Associating $V$ and $u(x)$ 
with the exterior products of their column vectors
this is equivalent to the case $k=1$.}
\end{proof}


\begin{thebibliography}{AAA}



\bibitem[Av1]{Avi} A. Avila, {\sl
Global Theory of one-frequency Schr\"odinger operators I: stratified analyticity of the Lyapunov
exponent and the boundary of non-uniform hyperbolicity}, Preprint 2009.

\bibitem[Av2]{Av2} A. Avila, {\sl
Density of positive Lyapunov exponents for $\SL(2,\R)$ cocycles}, 
Journal of the American Mathematical Society {\bf 24}, 999-1014
(2011). 

\bibitem[Av3]{A}A.  
Avila, {\sl KAM, Lyapunov exponents and the spectral dichotomy for
one-frequency Schr\"odinger operators.}  In preparation. 

\bibitem[ASV]{ASV}A. Avila, J. Santamaria, M. Viana, {\sl Cocycles
    over partially hyperbolic maps.} Asterisque, to appear.

\bibitem[AV]{AV}A. Avila, M. Viana, {\sl Extremal Lyapunov exponents:
    an invariance principle and applications.} Invent. Math. {\bf 181}
  115-189 (2010).

\bibitem[BV]{BV} J. Bochi, M. Viana
{\sl The Lyapunov Exponents of Generic Volume-Preserving and Symplectic Maps},
Annals of Math. {\bf 161}, 1423-1485 (2005)


\bibitem[B]{B}J. Bourgain, {\sl Positivity and continuity of the
    Lyapounov exponent for shifts on $T^d$ with arbitrary frequency
    vector and real analytic potential}  J. Anal. Math. {\bf 96},
  313-355 (2005).

\bibitem[BB]{bbook}J. Bourgain,{\sl Green's function estimates for
    lattice Schrödinger operators and applications.} Annals of
  Mathematics Studies, {\bf 158} Princeton University Press, Princeton, NJ, 2005.

\bibitem[BJ]{BJ}
J. Bourgain, S. Jitomirskaya,
{\sl Continuity of the Lyapunov exponent for quasiperiodic operators with analytic potential},
J. Statist. Phys. {\bf 108}, 1203-1218 (2002) 



\bibitem[CS]{CS}
W. Craig, B. Simon,
{\sl Subharmonicity of the Lyapunov index}, Duke Math. J. {\bf 50},
551-560 (1983)

\bibitem[DP]{DP}D. Dolgopyat, Y.Pesin, {\sl Every compact manifold carries
    a completely hyperbolic diffeomorphism}, Ergodic Theory and
  Dynamical Systems  {\bf 22}  409-437 (2002).

\bibitem[DK]{DK} Duarte, S. Klein {\sl Continuity of the Lyapunov
    exponents for quasiperiodic cocycles}. preprint 2013, to appear in Comm. Math. Phys.


\bibitem[GK]{gk} I. C. Gohberg and M. G. Krein, 
{\sl Introduction to the Theory of Linear Nonselfadjoint Operators}, 
American Mathematical Society, Translations of Mathematical Monographs, Vol. 18, Providence, RI, 1969.

\bibitem[GS]{GS}
M. Goldstein, W. Schlag,
{\sl H\"older Continuity of the Integrated Density of States for Quasi-Periodic Schrödinger},
Annals of Math. {\bf 154}, 155-203 (2001)

\bibitem[J]{j99} S. Jitomirskaya, {\sl Metal-insulator transition for
    the almost Mathieu operator.} Ann. of Math. (2) {\bf 150} 1159-1175 (1999).

\bibitem[J1]{J} S. Jitomirskaya, {\sl Analytic quasiperiodic matrix
    cocycles: questions of continuity and applications}, a minicourse
  at the conference  ''Recent Advances in Harmonic Analysis
and Spectral Theory'', Texas A\& M, August 6-10, 2012.




\bibitem[JKS]{JKS} S. Jitomirskaya,D. A. Koslover, M.S. Schulteis,
  {\sl Continuity of the Lyapunov exponent for analytic quasiperiodic
    cocycles} Ergodic Theory Dynam. Systems {\bf 29} 1881-1905 (2009).

\bibitem[JM]{jm}S. Jitomirskaya, C. Marx, {\sl Analytic quasi-perodic
    cocycles with singularities and the Lyapunov exponent of extended
    Harper's model.}  Comm. Math. Phys. {\bf 316} 237-267 (2012).

 
\bibitem[JM1]{JM1}
S. Jitomirskaya, C. Marx, {\sl
Analytic quasi-periodic Schr\"odinger operators and rational frequency approximants},
Geometric and Functional Analysis, {\bf 22},  1407-1443 (2012).

\bibitem[JM2]{jm2}
S. Jitomirskaya, C. Marx, {\sl
Continuity of the Lyapunov exponents for analytic quasiperiodic
cocycles with singularities,}
JFPTA, {\bf 10}, 129-146 (2011)

\bibitem[K]{K}I. V. Krasovsky, {\sl Bloch electron in a magnetic field
    and the Ising model,} Phys. Rev. Lett. {\bf 85} 4920-4923 (2000).


\bibitem[M]{M} R. Ma\~n\'e,
{\sl Oseledec's theorem from the generic viewpoint,}
Procs. Intern. Congress Math. Warszawa {\bf 2}, 1259-1276 (1983)

\bibitem[MP]{mp} P. M\"orters, Y. Peres, 
Brownian motion.
With an appendix by Oded Schramm and Wendelin Werner. Cambridge Series
in Statistical and Probabilistic Mathematics. Cambridge University
Press, Cambridge. xii+403 pp. (2010)

\bibitem[R]{R} D. Ruelle, {\sl Ergodic theory of differentiable
    dynamical systems}  Inst. Hautes Études Sci. Publ. Math. {\bf 50}
  , 27-58 (1979).

\bibitem[Sch]{Sch}
W. Schlag, {\sl 
Regularity and convergence rates for the Lyapunov exponents of linear co-cycles,}
preprint, arXiv:1211.0648 (2012)

\bibitem[T]{T}  K. Tao, {\sl  Continuity of Lyapunov exponent for
    analytic quasi-periodic cocycles on higher-dimensional torus.}
  Front. Math. China {\bf 7} 521-542 (2012). 

\bibitem[V]{V}M. Viana,  {\sl  Almost all cocycles over any hyperbolic
    system have nonvanishing Lyapunov exponents.} Ann. of Math. (2)
  {\bf 167}  643-680 (2008).

\bibitem[WY]{WY} Y. Wang, J. You, {\sl Examples of Discontinuity of
    Lyapunov Exponent in Smooth Quasi-Periodic Cocycles} Preprint 2012.

\end{thebibliography}
\end{document}